\documentclass[11pt]{article}
\usepackage[a4paper,margin=3cm]{geometry}
\usepackage{amsmath}
\usepackage{amssymb}

\newtheorem{theorem}{Theorem}[section]

\newtheorem{corollary}[theorem]{Corollary}
\newtheorem{proposition}[theorem]{Proposition}
\newenvironment{proof}{{\bf Proof.}}{\hfill$\Box$\\}
\newenvironment{remark}{{\vskip 1ex\noindent\bf Remark.}}{\\}

\newenvironment{definition}{{\vskip 1ex\noindent\bf Definition.}}{\\}

\newcommand{\R}{\mathbb{R}}

\newcommand{\tr}{\mathrm{tr}}
\newcommand{\Id}{\mathrm{Id}}
\newcommand{\const}{\mathrm{const}}
\newcommand{\spn}{\mathrm{span}}

\newcommand{\FF}{\mathcal{F}}
\newcommand{\DD}{\mathcal{D}}

\newcommand{\VV}{\mathcal{V}}
\newcommand{\AAA}{\mathcal{A}}
\newcommand{\nablac}{{\nabla}^C}

\title{{\bf Webs and the Pleba\'nski equation}}
\author{
Wojciech Kry\'nski\thanks{
{\bf Institute of Mathematics, Polish Academy of Sciences, ul.~\'Sniadeckich 8, 00-956 Warszawa, Poland}
\newline 
E-mail: krynski@impan.pl
}
}

\begin{document}
\maketitle
\begin{abstract}
We consider 3-webs, hyper-para-complex structures and integrable Segre structures on manifolds of even dimension and generalise the second heavenly Pleba\'nski equation in the context of higher-dimensional hyper-para-complex structures. We also characterise the Segre structures admitting a compatible hyper-para-complex structure in terms of systems of ordinary differential equations.
\end{abstract}

\section{Introduction}\label{sec1}

\paragraph{Results.}
A web on a manifold is a family of foliations. In the present paper we consider classical 3-webs on $2m$-dimensional manifolds and exploit connections between them, hyper-para-complex structures and $\alpha$-integrable $(m,2)$-Segre structures. In dimension 4 the structures coincide with the hyper-Hermitian metrics of the neutral signature. The geometry of webs is an interesting subject on its own (see the survey \cite{AG}). Webs are also immanently connected to integrable systems and naturally appear as underlying structures in various problems in differential geometry \cite{DK1,GZ,K2,Z}. For instance, Veronese webs in dimension 3 can be used as an alternative description of the hyper-CR Einstein-Weyl structures \cite{D2,DK1} (see also \cite{Ca,DK2,FK,Z} and \cite{K3} for higher-dimensional generalisations), or 3-webs on a plane give rise to the projective structures with skew-symmetric Ricci tensors \cite{K2} (see also \cite{De,R}). The webs provide a very useful description of the structures.

Our main aim in the present paper is to generalise the Pleba\'nski equation \cite{P} to higher dimensions in the context of the integrable Segre structures. We use our earlier results on webs \cite{K1} and prove that any solution $\theta=(\theta_{ij})$, $\theta_{ij}\colon\R^{2m}\to\R$, $i,j=1,\ldots,m$, $i<j$ to the overdetermined system
$$
\begin{aligned}
\sum_{p,q=1}^m(-1)^{p+q}\left((\partial_{x_i}\partial_{x_p}\theta_{pl})(\partial_{x_j}\partial_{x_q}\theta_{qk}) -(\partial_{x_j}\partial_{x_p}\theta_{pl})(\partial_{x_i}\partial_{x_q}\theta_{qk})\right)=&\\ =\partial_{y_j}\partial_{x_i}\theta_{kl}-\partial_{y_i}\partial_{x_j}\theta_{kl},\qquad i<j,\ k<l,
\end{aligned}
$$
give rise to a Ricci-flat hyper-para-complex structure on $\R^{2m}$. More detailed formulation of the result is contained in Section \ref{sec3} (Theorem \ref{thm3}) and there is one additional term that depends on arbitrary functions in the full version of the system. Our main tool is the Chern connection of webs \cite{N}. The Chern connection in dimension 4 is a canonical Weyl connection uniquely associated to a hyper-Hermitian metric via a fibration of the associated twistor space \cite{DW2,Pe}. 

In the second part of the paper we consider systems of second order ordinary differential equations (ODEs). However, we do not use point transformations as in \cite{CDT,G} but consider more rigid notion of equivalence that essentially goes back to Chern \cite{C}. As a consequence we are able to state a correspondence between classes of ODEs and hyper-Hermitian or hyper-K\"ahler structures in dimension 4 (Theorem \ref{thm4}). A characterisation of the hyper-Hermitian and the hyper-K\"ahler structures in terms of the point invariants of ODEs is an outstanding open problem. We conclude the paper with a direct construction of bi-Hamiltonian structures from $\alpha$-integrable Segre structures (this is a construction analogous to a construction proposed in \cite{DK1} in dimension 3).

\paragraph{The plan of the paper.}
In the next section we briefly describe all geometric structures that appear later on. In particular we consider: 3-webs, Segre structures, hyper-para-complex structures and Kronecker webs. The geometry behind all these structures is similar, however there are subtle, important differences. Moreover, there are various non-equivalent definitions of the structures in the literature. Therefore, we clarify the terminology adapted in the present paper (which is close to \cite{AMT}) and emphasize the differences among the structures. 

Section \ref{sec3} contains our generalisation of the Pleba\'nski equation. In Section \ref{sec4} we consider systems of ODEs. Finally, Section \ref{sec5} is an appendix containing useful formulae for the torsion and the curvature of the Chern connection in coordinates adapted to webs.

\section{Geometric structures}\label{sec2}

\paragraph{3-webs and the Chern connection.} 
A 3-web $\{\FF_1,\FF_2,\FF_3\}$ on a manifold $M$ of dimension $2m$ is a triple of $m$-dimensional foliations in the general position meaning that any two of them intersect transversally at each point of $M$. One associates the Chern connection to a 3-web. It is defined by the following formula (see \cite{N})
\begin{equation}\label{chern_def}
\begin{aligned}
\nablac_XY=&\pi_h(j[\pi_h(X),j\pi_h(Y)]+[\pi_v(X),\pi_h(Y)])\\
&+\pi_v(j[\pi_v(X),j\pi_v(Y)]+[\pi_h(X),\pi_v(Y)])
\end{aligned}
\end{equation}
where $\pi_h\colon TM\to T\FF_2$ is the projection to $T\FF_2$ along $T\FF_1$, $\pi_v\colon TM\to T\FF_1$ is the projection to $T\FF_1$ along $T\FF_2$ and $j\colon TM\to TM$ is a vector bundle isomorphism defined uniquely by
\begin{equation}\label{end_j_form}
\begin{aligned}
j\pi_h(X)\in T\FF_1,&\qquad j \pi_h(X)-\pi_h(X)\in T\FF_3, \\
j\pi_v(X)\in T\FF_2,&\qquad j\pi_v(X)-\pi_v(X)\in T\FF_3
\end{aligned}
\end{equation}
for any $X\in TM$. One can verify that the Chern connection is the unique affine connection on $M$ such that
\begin{enumerate}
\item[(C1)] all leaves of $\FF_1$, $\FF_2$ and $\FF_3$ are parallel with respect to $\nablac$,
\item[(C2)] torsion $T(\nablac)(X,Y)$ is normalised such that it vanishes for all $X$ tangent to $\FF_1$ and $Y$ tangent to $\FF_2$.
\end{enumerate}
Note that $\{\FF_1,\FF_2,\FF_3\}$ reduces the bundle of all frames on $M$ to a $GL(m)$-bundle (e.g. fixing a frame in $T\FF_1$ one uses $j$ to extend it to a frame in $T\FF_2$ and gets a frame in $TM$).

\paragraph{$(m,2)$-Segre structures.}
A $(m,n)$-Segre structure on a manifold $M$ of dimension $mn$ is a smooth field of cones $S(x)\subset T_xM$ in the tangent spaces of $M$, each linearly isomorphic to the Segre cone of linear maps $\R^m\to\R^n$ of rank one \cite{M}.  The structures are sometimes called almost-Grassmann structures and we refer to \cite{AG0,G} for more information on their geometry. In the present paper we consider the $(m,2)$-Segre structures. Then, any tangent space $T_xM$ is identified with $\R^m\otimes\R^2$. The decomposition defines a one-parameter family of special $m$-dimensional subspaces of $T_xM$, called $\alpha$-planes, by the formula $\R^m\otimes v$, where $v\in\R^2\setminus\{0\}$, at each point $x\in M$ (analogously one has $(m-1)$-parameter family of $\beta$-planes $w\otimes\R^2$, where $w\in\R^m\setminus\{0\}$).

An $\alpha$-submaniofld of $M$ is an $m$-dimensional submanifold of $M$ such that its all tangent spaces are $\alpha$-planes. A Segre structure is $\alpha$-integrable if all $\alpha$-planes are tangent to $\alpha$-submanifolds.

In the particular case of $m=2$ the Segre structures coincide with the conformal metrics of neutral signature $(2,2)$ on 4-dimensional manifolds. The seminal result of Penrose \cite{Pe} states that the $\alpha$-integrability in this case is equivalent to the fact that a structure is anti-self-dual. We refer to \cite{DW2} for more information on the anti-self-dual structures in the neutral signature.

\paragraph{Hyper-para-complex structures.}
A hyper-para-complex structure on a manifold $M$ of dimension $2m$ is a triple of anti-commuting endomorphism $I,K,J\colon TM\to TM$ such that $J$ is a complex structure, $I$ and $K$ are para-complex structures and
$$
IK=-KI=J.
$$
We refer to \cite{AMT} for more information on the structures. Let us point out, that it follows from the definition that at each point $x\in M$ both $I$ and $K$ have $m$-dimensional eigen-spaces corresponding to $1$ and $-1$, denoted $\DD_I^+(x)$, $\DD_I^-(x)$ and $\DD_K^+(x)$, $\DD_K^-(x)$, respectively. All distributions $\DD_I^\pm$ and $\DD_K^\pm$ are integrable.

\paragraph{Hyper-para-complex Segre structures.}
Let $M$ be a manifold of dimension $2m$. We shall say that a hyper-para-complex structure $(I,K,J)$ is compatible with a $(m,2)$-Segre structure $S$ if all $I$, $J$ and $K$ preserve the Segre cones $S(x)$, $x\in M$. The key ingredient of the following theorem is \cite[Theorem 4.14]{N} (see also \cite{AG0}).
\begin{theorem}\label{thm1}
Let $\{\FF_1,\FF_2,\FF_3\}$ be a 3-web on a manifold $M$ of dimension $2m$. Then there is a unique $(2,m)$-Segre structure $S$ on $M$ such that the tangent spaces to all $\FF_i$, $i=1,2,3$, are $\alpha$-planes of $S$. Moreover, if the torsion of the Chern connection $\nablac$ vanishes then $S$ is $\alpha$-integrable and admits a compatible hyper-para-complex structure. Conversely, all hyper-para-complex structures compatible with a Segre structure on $M$ can be locally obtained in this way. 
\end{theorem}
\begin{proof}
The existence of $S$ follows from the fact that three points in the Grasmannian $G_m(\R^{2m})$ of $m$-planes in $\R^{2m}$ determine uniquely a Veronese embedding of the projective curve $\R P^1$ into $G_m(\R^{2m})$. The image of the embeding is a Segre cone. If a 3-web is given, then the construction can be applied to $G_m(T_xM)$ at any point $x\in M$, where one takes $T_x\FF_i$, $i=1,2,3$, as the three points in the Grasmannian. The Segre cone $S$ is a collection of all vectors of the form
$$
sX-tj(X)
$$
where $X\in T\FF_1$, $(s:t)\in\R P^1$ is arbitrary and $j$ is an endomorphism defined by \eqref{end_j_form}. We shall denote  $\VV_{(s:t)}=\spn \{sX-tj(X)\ |\ X\in T\FF_1\}$  and abbreviate $\VV_t=\VV_{(1:t)}$ and $\VV_\infty=\VV_{(0:1)}$. All $\VV_t$, $t\in \R$, are rank-$m$ distributions on $M$. Moreover, the tangent planes $T_x\VV_t$, $x\in M$, $t\in\R$, are exactly the $\alpha$-planes of $S$. Note that $\VV_0=T\FF_1$, $\VV_\infty=T\FF_2$ and $\VV_1=T\FF_3$ are integrable by definition. The integrability of $\VV_t$, where $t\in\R$ is arbitrary, is equivalent to the vanishing of the torsion of the Chern connection $\nablac$ (see \cite[Theorem 4.14]{N}). Therefore if the torsion of $\nablac$ vanishes then $S$ admits a $(m+1)$-parameter family of $\alpha$-submanifolds in $M$ defined by all leaves of all distributions $\VV_t$, $t\in\R$. It follows that $S$ is $\alpha$-integrable. In particular $\VV_{-1}=\spn\{sX+tj(X)\ |\ X\in\VV_0\}$ is integrable. Therefore, the two decompositions
$$
TM=\VV_0\oplus\VV_\infty,\qquad TM=\VV_1\oplus\VV_{-1}
$$
define two para-complex structures $I$ and $K$ such that $\DD_I^+=\VV_0$, $\DD_I^-=\VV_\infty$ and $\DD_K^+=\VV_1$, $\DD_K^-=\VV_{-1}$. We shall prove that $I$ and $K$ preserve $S$. Let $X\in\VV_0$. Then $j(X)\in\VV_\infty$. Therefore $I(X-tj(X))=X+tj(X)$. Hence $I(\VV_t)=\VV_{-t}$. Moreover we have $X-tj(X)=\frac{1+t}{2}(X-j(X))+\frac{1-t}{2}(X+j(X))$, where $X-j(X)\in\VV_1$ and $X+j(X)\in\VV_{-1}$. Therefore $K(X-tj(X))=tX-j(X)$. Hence $K(\VV_t)=\VV_{1/t}$. Similarly one proves that the structures anti-commute and $J=IK$ is a complex structure complementing $I$ and $K$ to a hyper-para-complex structure. Indeed we have $I(X)=X$, $I(j(X))=-j(X)$, $K(X)=-j(X)$ and $K(j(X))=-X$. Thus $IK(X)=j(X)=-KI(X)$ and $IK(j(X))=-X=-KI(j(X))$. This proves the first part of the theorem.

In order to prove the converse statement we note that if $S$ is a Segre structure and $(I,J,K)$ is a compatible hyper-para-complex structure then we can locally define a 3-web $\{\FF_1,\FF_2,\FF_3\}$ by $T\FF_1=\DD_I^+$, $T\FF_2=\DD_I^-$ and $T\FF_3=\DD_K^+$. Note that $\DD_K^-$ is defined as the span of all $I(X)$ with $X\in\DD_K^+$. Hence the hyper-para-complex structure is determined uniquely by $\{\FF_1,\FF_2,\FF_3\}$.
\end{proof}

The space of all $\alpha$-submanifolds of an $\alpha$-integrable Segre structure $S$ is called the twistor space of $S$ and will be denoted $\mathcal{T}(S)$. It follows from Theorem \ref{thm1} that in the presence of a hyper-para-complex structure the twistor space has a natural fibration over $\R P^1$ given by the projective parameter $(s:t)$ defined by the corresponding 3-web in the proof of Theorem \ref{thm1}. The fibration
$$
\mathcal{T}(S)\to\R P^1
$$ 
is such that for any $(s:t)\in\R P^1$ the $\alpha$-submanifolds in the fibre $\mathcal{T}(S)_{(s:t)}$ form a foliation of $M$. Thus, there is a one-parameter family of foliations $\FF=\{\FF_{(s:t)}\}_{(s:t)\in\R P^1}$. In terms of the distributions $\VV_t=\spn \{X-tj(X)\ |\ X\in T\FF_1\}$ the foliations are given by
\begin{equation}\label{web_form}
T\FF_t= \VV_t
\end{equation}
where, as before, we use an affine parameter $t=(1:t)$, for convenience. It will be useful to adapt the following definition later on.
\begin{definition}
An integrable Segre structure $S$ with a fixed fibration $\mathcal{T}(S)\to\R P^1$ of the twistor space as above is called a \emph{hyper-para-complex Segre structure}.
\end{definition}

It is an easy observation that the Chern connection for a torsion-free 3-web $\{\FF_1,\FF_2,\FF_3\}$ does not depend on the ordering of the foliations, because if the torsion vanishes then the normalisation condition (C2) is fulfilled with respect to any order of the foliations. Moreover, if $T(\nablac)=0$ then the Chern connection of $\{\FF_1,\FF_2,\FF_3\}$ coincides with the Chern connection of a 3-web $\{\tilde\FF_1,\tilde\FF_2,\tilde\FF_3\}$, where $T\tilde{\FF_i}=\VV_{t_i}$ for three distinct points $t_1,t_2,t_3\in\R$. This follows from the fact that all $\VV_t$ are parallel with respect to $\nablac$, i.e. (C1) is satisfied independently of the choice of the three points in $\R P^1$. Hence, the Chern connection is canonically assigned to a hyper-para-complex Segre structure via the fibration of the corresponding twistor space. If $m=2$ we get the following

\begin{corollary}\label{cor1}
Let $[g]$ be a conformal metric of neutral signature defined by a 3-web on a four-dimensional manifold. Then the corresponding Chern connection is a Weyl connection for $[g]$.
\end{corollary}
\begin{proof}
Indeed, the Chern connection preserves the null cone of $[g]$, i.e.~it preserves the conformal class itself.
\end{proof}

\begin{remark}
In the context of hyper-Hermitian structures the Chern connection is usually called the Obata connection, see \cite{Ca}.
\end{remark}

\paragraph{Kronecker webs.}
A Kronecker web on a manifold $M$ is a 1-parameter family of foliations $\FF=\{\FF_{(s:t)}\}_{(s:t)\in\R P^1}$ of fixed codimension $k$ such that in a neighbourhood of any $x\in M$ one can find $(s:t)$-dependent one-forms $\alpha^1_{(s:t)},\ldots,\alpha^k_{(s:t)}$ such that
$$
T\FF_{(s:t)}=\bigcap_{i=1,\ldots,k}\ker\alpha^i_{(s:t)}
$$
and all $\alpha^i_{(s:t)}$ are polynomial in $(s:t)$ and point-wise linearly independent for any fixed $(s:t)\in\R P^1$. The Kronecker webs were introduced in \cite{Pa, Z1} and appear as reductions of certain bi-Hamiltonian structures. If $k=1$ then the webs are called Veronese webs and this particular case was initially studied by Gelfand and Zakharevich \cite{GZ}. We refer to \cite{GZ, Pa, Z1} for more information on the webs.

The family $\FF$ defined by \eqref{web_form} is an example of a Kronecker web. Indeed, in this case $\dim M=2m$, $k=m$ and one sees that the tangent bundles $T\FF_{(s:t)}$ are annihilated by one-forms linear in $(s:t)$. The webs of this form are called isotypic Kronecker webs of order one in the terminology of \cite{K1}.

Two Kronecker webs $\{\FF_t\}$ and $\{\FF'_t\}$ on $M$ are said to be equivalent in \cite{K1} if there exists a diffeomorphism $\varphi$ such that $\varphi(\FF_t)=\FF'_t$ for any $t=(1:t)$. This means that the affine parameter $t$ as well as the projective one $(s:t)$ are chosen and canonically assigned to a web. In particular the two decompositions $T\FF_0\oplus T\FF_\infty$ and $T\FF_1\oplus T\FF_{-1}$ of $TM$ are canonically given and define a canonical hyper-para-complex structure assigned to the Kronecker web. Conversely, any hyper-para-complex structure defines a family $\{\FF_t\}$ as in the proof of Theorem \ref{thm1}. From this point of view, the isotypic Kronecker webs of order one and the hyper-para-complex structures are equivalent objects. One can also say that a Kronecker web is a hyper-para-complex Segre structure with a fixed parametrisation of $\R P^1$ in the fibration $\mathcal{T}(S)\to \R P^1$. The later are also equivalent to the torsion-free 3-webs because the choice of three points in $\R P^1$ fixes a parametrisation of $\R P^1$ as in the proof of Theorem \ref{thm1}.

\begin{definition}
A Kronecker web, or a hyper-para-complex structure is called \emph{Ricci-flat} if the Ricci curvature tensor of the corresponding torsion-free Chern connection vanishes. 
\end{definition}

In the next section we shall use the correspondence between the isotypic Kronecker webs and hyper-para-complex structures and interpret results of \cite{K1} in the context of the hyper-para-complex structures.

\section{Pleba\'nski's equations}\label{sec3}

\paragraph{Normal form of hyper-para-complex structures.}
We shall show now how one can generalise the second heavenly Pleba\'nski equation \cite{P} (see also \cite{FP,KM}) to higher dimensional structures. The first step is the following reformulation of our result \cite[Corollary 7.5]{K1}.
\begin{theorem}\label{thm2}
There is a one to one correspondence between germs of hyper-para-complex structures on $\R^{2m}$ and germs of functions $R^1\ldots,R^m\colon\R^{2m}\to\R$ given modulo transformations
$$
(R^1(x,y),\ldots,R^m(x,y))^T\mapsto G\cdot (R^1(G^{-1}x,G^{-1}y),\ldots,R^m(G^{-1}x,G^{-1}y))^T, \qquad x,y\in\R^m,
$$
where $G\in GL(m,\R)$ is a constant matrix, and satisfying
\begin{equation}\label{eq1}
\sum_{l=1,\ldots,m}\left(\partial_{x_i}R^l\partial_{x_l}\partial_{x_j}R^k-\partial_{x_j}R^l\partial_{x_l}\partial_{x_i}R^k\right)
= \partial_{y_j}\partial_{x_i}R^k-\partial_{y_i}\partial_{x_j}R^k
\end{equation}
and
\begin{equation}\label{eq2}
R^k(0,y)=\partial_{x_i}R^k(0,y)=\partial_{x_i}\partial_{x_j}R^k(0,y)=0
\end{equation}
for all $i,j,k=1,\ldots,m$, where $(x,y)=(x_1,\ldots,x_m,y_1,\ldots,y_m)$ are suitable coordinates on $\R^{2m}$. Moreover $R_{ijl}^k=\partial_{x_i}\partial_{x_j}\partial_{x_l}R^k$ are components of the Riemann curvature tensor of the Chern connection $\nablac$ and $Ric(\nablac)=0$ if and only if $\sum_{k=1}^mR_{ijk}^k=0$.
\end{theorem}
\begin{proof}
It is proved in \cite[Corollary 7.5]{K1} that there is a one to one correspondence between germs of the isotypic Kronecker webs and germs of functions $R=(R^1,\ldots,R^m)$ as above. In view of the discussion in the previous section this result can be directly reformulated in terms of the hyper-para-complex structures. 
\end{proof}

By Theorem \ref{thm1} the hyper-para-complex structures extend to $\alpha$-integrable Segre structures. The $\alpha$-submanifolds are tangent to 
\begin{equation}\label{eq_Lax}
\VV_t=\spn\left\{\partial_{x_i}+t\partial_{y_i}+t\sum_{j=1,\ldots,m}\partial_{x_i}R^j\partial_{x_j}\ |\ i=1,\ldots,m\right\}.
\end{equation}
where $t\in\R$. The vector fields $L_i(t)=\partial_{x_i}+t\partial_{y_i}+t\sum_{j=1,\ldots,m}\partial_{x_i}R^j\partial_{x_j}$ commute if and only if \eqref{eq1} holds. $(L_1,\ldots,L_m)$ will be referred to as the Lax tuple of the Segre structure.

\paragraph{Higher heavenly equations.}
We shall consider the Ricci-flat structures, i.e. we assume that $Ric(\nablac)=\sum_{k=1}^mR_{ijk}^k=0$. However, $R_{ijl}^k=\partial_{x_i}\partial_{x_j}\partial_{x_l}R^k$ and using the boundary condition \eqref{eq2} we can integrate the equation $\sum_{k=1}^mR_{ijk}^k=0$ and get
$$
\sum_{i=1,\ldots,m}\partial_{x_i}R^i=0,
$$
which is equivalently written as
$$
d_x\rho=0,
$$
where
$$
\rho=\sum_{i=1,\ldots,m}(-1)^iR^idx_1\wedge\ldots\wedge dx_{i-1}\wedge dx_{i+1}\wedge\ldots\wedge dx_m=0
$$
and $d_x$ denotes the exterior differential with respect to $x$-coordinates only. This implies that there exist functions $\theta_{ij}$, $i<j$, such that
$$
\rho=d_x\left(\sum_{i<j}\theta_{ij}\bigwedge_{\substack{s=1,\ldots,m\\ s\neq i,j}}dx_s\right).
$$
In particular if one denotes $\theta_{ij}=-\theta_{ji}$ for $i>j$, then 
\begin{equation}\label{eq_R}
R^i=\sum_{\substack{j=1,\ldots,m\\j\neq i}}(-1)^{i+j}\partial_{x_j}\theta_{ji}.
\end{equation}
Now we rewrite \eqref{eq1} in terms of $\theta_{ij}$. Changing the order of differentiations, we get
\begin{equation}\label{eq_a}
d_x\gamma_{ij}=d_x\beta_{ij}
\end{equation}
where
\begin{equation}\label{eq_gamma}
\gamma_{ij}=\sum_{l<k}(\partial_{x_i}R^l\partial_{x_j}R^k-\partial_{x_j}R^l\partial_{x_i}R^k)\bigwedge_{\substack{s=1,\ldots,m\\ s\neq l,k}}dx_s
\end{equation}
and
\begin{equation}\label{eq_beta}
\beta_{ij}=\sum_{l<k}(\partial_{y_j}\partial_{x_i}\theta_{kl}-\partial_{y_i}\partial_{x_j}\theta_{kl})\bigwedge_{\substack{s=1,\ldots,m\\ s\neq l,k}}dx_s.
\end{equation}
It follows from the Poincare lemma applied to \eqref{eq_a} that locally there exist $(m-3)$-forms $f_{ij}$ such that
\begin{equation}\label{eq_pleb}
\gamma_{ij}=\beta_{ij}+d_xf_{ij}.
\end{equation}
Additionally we can assume
\begin{equation}\label{eq_pleb_cond}
f_{ij}(0,y)=0,
\end{equation}
which follows from \eqref{eq2}. We have proved the following
\begin{theorem}\label{thm3}
Let
$$
f_{ij}=\sum_{p<q<r}f_{ij}^{pqr}\bigwedge_{\substack{s=1,\ldots,m\\ s\neq p,q,r}}dx_s
$$
be a collection of germs at 0 of $(m-3)$-forms on $\R^{2m}$ satisfying \eqref{eq_pleb_cond}. Then, any solution $\theta=(\theta_{ij})$ to \eqref{eq_pleb} give rise to a germ of a Ricci-flat hyper-para-complex structure via \eqref{eq_R} and \eqref{eq_Lax}. Conversely, any germ of a Ricci-flat hyper-para-complex structure can be obtained in this way.
\end{theorem}
Equation \eqref{eq_pleb} is a generalisation of the second heavenly Pleba\'nski's equation \cite{P}. Indeed, if $m=2$ then there is only one function $\theta=\theta_{12}$, there are no functions $f_{ij}$ and \eqref{eq_pleb} takes the form
$$
(\partial_{x_1}^2\theta)(\partial_{x_2}^2\theta)-(\partial_{x_1}\partial_{x_2}\theta)^2= \partial_{y_1}\partial_{x_2}\theta-\partial_{y_2}\partial_{x_1}\theta.
$$
In the general case, any $(m-3)$-form $f_{ij}$ has $\binom{m}{3}$ components and $(m-2)$-forms $\beta_{ij}$ and $\gamma_{ij}$ have $\binom{m}{2}$ components each. Therefore we have $\binom{m}{3}\binom{m}{2}$ arbitrary functions in the system \eqref{eq_pleb}, which consists of $\binom{m}{2}\binom{m}{2}$ equations for $\binom{m}{2}$ functions $\theta_{ij}$. Explicitly
$$
\begin{aligned}
\sum_{p,q=1}^m(-1)^{p+q}\left((\partial_{x_i}\partial_{x_p}\theta_{pl})(\partial_{x_j}\partial_{x_q}\theta_{qk}) -(\partial_{x_j}\partial_{x_p}\theta_{pl})(\partial_{x_i}\partial_{x_q}\theta_{qk})\right)=&\\ =\partial_{y_j}\partial_{x_i}\theta_{kl}-\partial_{y_i}\partial_{x_j}\theta_{kl}
+(df_{ij})_{kl},
\end{aligned}
$$
where $i<j$, $k<l$ and $(df_{ij})_{kl}$ is a coefficient of $df_{ij}$ next to $\bigwedge_{s=1,\ldots,m,\ s\neq l,k}dx_s$.
The system is overdetermined for $m>2$. In particular, in the case $m=3$ we have 9 equations (depending on 3 arbitrary functions) for 3 unknown functions.

\section{Systems of ordinary differential equations}\label{sec4}

\paragraph{A class of transformations.}
We shall consider systems of second order ODEs given in the form
$$
x''=F(t,x,x'),\eqno{(F)}
$$
where $x=(x_1,\ldots,x_m)\in\R^m$ and $t\in\R$. The geometry of such systems is related to the torsion-free path geometries \cite{G}. In particular, if $m=2$ then the main result of \cite{G} can be rephrased that there is a correspondence between anti-self-dual metrics of split signature on $\R^4$ and pairs of second order ODEs given up to point transformations for which the so-called Wilczynski invariant vanishes (see \cite{CDT}). However, the problem of characterisation of the classes of point equivalent ODEs that correspond to hyper-Hermitian or hyper-K\"ahler structures remains open. The similar problem for the hyper-CR Einstein-Weyl structures in dimension 3 has been solved in \cite{DK2}, but the solution is complicated and leads to invariants of very high order. In the present case the problem seems to be even more difficult. However, it has a much simpler solution if one uses \cite{K1} and considers ODEs modulo so-called Veronese transformations. This is a more rigid class of transformations in comparison to point transformations. In order to provide the definition, we recall that a system $(F)$ can be encoded by the total derivative vector field $X_F$ defined on the space of jets. In the present case we consider $J^1(\R,\R^m)$ the space of 1-jets of mappings $\R\to\R^m$ with natural coordinates $(t,x_1,\ldots,x_m,y_1,\ldots,y_m)$, where $y_i$ corresponds to $\frac{d x_i}{dt}$. Then
$$
X_F=\partial_t+\sum_{i=1,\ldots,m}\left(y_i\partial_{x_i}+F_i\partial_{y_i}\right).
$$

We say that $\Phi\colon J^1(\R,\R^m)\to J^1(\R,\R^m)$ is a Veronese transformation if it is a point transformation such that the independent variable is transformed by the formula
\begin{equation}\label{trans_form}
t\mapsto \frac{at+b}{ct+d}
\end{equation}
where $a,b,c,d\in\R$ are constant and $ad-bc\neq 0$.

\paragraph{ODEs and hyper-pseudo-complex structures.}
There are two basic invariants of $(F)$ modulo Veronese transformations. The first one is given by the matrix
$$
T=\left(-\partial_{x_i}F_j+ \frac{1}{2}X_F(\partial_{y_i}F_j)- \frac{1}{4}\sum_{k=1,\ldots,m}\partial_{y_i}F_k\partial_{y_k}F_j\right)_{i,j=1,\ldots,m}
$$
which can be interpret as a vector bundle endomorphism of $\VV=\spn\{\partial_{y_1},\ldots,\partial_{y_m}\}$. It is sometimes called the Jacobi endomorphism. In the context of ODEs it has already appeared in the paper of Chern \cite[problem (B)]{C}, who considered even more rigid transformations, such that $t$ remains constant. However, it is a matter of simple calculations that $T$ is preserved also by \eqref{trans_form} (indeed if $t\mapsto \varphi(t)$ then $T$ is modified by $\mathcal{S}(\varphi)\Id$, where $\mathcal{S}(\varphi)$ is the Schwartzian derivative of $\varphi$). The invariant has been generalised to a class of control systems in \cite{JK}, where it is called the curvature operator. Moreover, the Wilczynski invariant, which is a point invariant, is the trace-free part of $T$ (see \cite{G,CDT}). The second invariant of $(F)$ is given by
$$
B=\left(-\frac{1}{2}\partial_{y_i}\partial_{y_j}\partial_{y_k}F^l\right)_{i,j,k,l=1,\ldots,m}
$$
and can be thought of as a $(3,1)$-tensor on $\VV$ (see \cite[Definition 4.10]{JK}) and is called the Berwald curvature in the context of the Finsler geometry. We shall denote
$$
\tr\,B=\left(-\frac{1}{2}\sum_{k=1,\ldots,m}\partial_{y_i}\partial_{y_j}\partial_{y_k}F^k\right)_{i,j=1,\ldots,m}.
$$

\begin{theorem}\label{thm4}
There is a one to one correspondence between systems of $m$ second order ODEs given up to Veronese transformations such that
$T=0$ and hyper-para-complex $(m,2)$-Segre structures on $\R^{2m}$. The Ricci-flat structures correspond to equations with $\tr\,B=0$.
\end{theorem}
\begin{proof}
The first part is a reformulation of \cite[Theorem 5.3]{K1} where we have proved a correspondence between ODEs and the Kronecker webs of the form $\FF=\{\FF_t\}$. However, in \cite{K1} we have considered the transformations preserving the independent variable $t$ as in \cite[Problem (B)]{C}. Here, the Veronese transformations $t\mapsto \frac{at+b}{ct+d}$ are needed, because we work with the Segre structures instead of the Kronecker webs (c.f. the discussion in Section \ref{sec2}). 

In order to prove the second part we shall show that the condition $\tr\,B=0$ is equivalent to the fact that the Chern connection for a Segre structure is Ricci flat. This follows from the structure equations for canonical frames associated to canonical connections for ODEs \cite[Theorem 4.8]{JK} and for webs \cite[Theorem 7.1]{K1}. Indeed, the connection defined in \cite{K1} coincides with the Chern connection (because it satisfies (C1) and (C2), i.e. it is torsion-free and preserves foliations $\FF_t$). On the other hand, it follows from the structure equations in \cite{JK} that if $T=0$ then on each slice $\{t=\const\}$ the connection for system $(F)$ can be identified with the Chern connection on $J^1(\R,\R^m)/X_F$. Under this identification $\tr\,B$ equals to $Ric(\nablac)$ (in particular $\mathcal{L}_{X_F}(B)=0$ provided that $T=0$, where $\mathcal{L}_{X_F}$ stands for the Lie derivative).
\end{proof}

\begin{remark}
The condition $\tr\,B=0$ is equivalent to the fact that $\sum_{k=1}^m \partial_{y_k}F^k$ is linear in $y_i$'s. One can verify that this condition is satisfied for all examples of \cite{CDT} corresponding to the hyper-K\"ahler metrics.
\end{remark}

\paragraph{Point transformations.}
There is an outstanding open question how to characterise the classes of point-equivalent systems of ODEs with the vanishing Wilczynski invariant that correspond to the hyper-Hermitian structures. In view of Theorem \ref{thm4} the question reduces to the problem of determining if there is a system with $T=0$ in a class of point equivalent ODEs.

Proceeding as in \cite{DK2} and computing the transformation rule for $T$ one gets the following

\begin{proposition}
A system $(F)$ with the vanishing Wilczynski invariant is point equivalent to a system with $T=0$ if and only if the equation
\begin{equation}\label{eq_wilcz}
g^2\tr\,T=-gX_F^2(g)+\frac{3}{2}X_F(g)^2
\end{equation}
has a solution $g\colon J^1(\R,\R^2)\to\R$ of the form $g=\partial_t\tilde t(t,x_1,x_2)+y_1\partial_{x_1}\tilde t (t,x_1,x_2)+y_2\partial_{x_2}\tilde t(t,x_1,x_2)$, where $\tilde t=\tilde t(t,x_1,x_2)$ is a function such that $\partial_t\tilde t\neq 0$.
\end{proposition}
Unfortunately we are unable to find integrability conditions for \eqref{eq_wilcz}. In the hyper-CR case \cite{DK2} the situation was easier since the equation similar to \eqref{eq_wilcz} was polynomial in a variable corresponding to the second derivative of the dependent variable and differentiating the equation sufficiently many times with respect to this variable one could get a first order PDE which was easy to solve. One does not have a similar phenomenon in the present case. 

\paragraph{Bi-Hamiltonian structures.}
As mentioned earlier the Kronecker webs arise as reductions of bi-Hamiltonian structures, understood as pairs of compatible Poisson structures. We shall present a simple procedure of reconstructing the bi-Hamiltonian structures using the Lax tuples.

Let $(L_1,\ldots,L_m)$ be the Lax tuple for a hyper-para-complex $(m,2)$-Segre structure given by formulae \eqref{eq_Lax}. We introduce $m$ new variables $p_1,\ldots,p_m$ and define $t$-depended Poisson structure on $M\times\R^m$ by the formula
$$
P(t)=L_1\wedge\partial_{p_1}+L_1\wedge\partial_{p_1}+\ldots+L_m\wedge\partial_{p_m}.
$$
(Actually, in the language of the Segre structures we have $TM=V_m\otimes V_2$, where $V_m$ and $V_2$ are vector bundles of rank $m$ and $2$, respectively, and then $(p_1,\ldots,p_m)$ are linear coordinates on $V_m^*\to M$.) It is a matter of calculations to verify that the formula for $P(t)$ defines a 1-parameter family of compatible Poisson structures. The structures are degenerate and decompose into $m$ Kronecker blocks of dimension 3 each. Any such structure can be locally obtained in this way.

\section{Appendix: adapted coordinates}\label{sec5}

\paragraph{Torsion and curvature.}
We shall present here yet another description of the 3-webs. Namely, we shall describe 3-webs and the related objects in a convenient coordinate system. We restrict ourselves to the most interesting case of dimension 4, however all formulae can be easily generalise to higher dimensions just by extending the range of indices.

Let $\{\FF_1,\FF_2,\FF_3\}$ be a 3-web on $M$ of dimension $4$. In a neighbourhood of any point in $M$ we can find functions $x_i$, $y_i$ and $w_i$, $i=1,2$, such that the leaves of $\FF_1$, $\FF_2$ and $\FF_3$ are locally defined by  $\{x_i=\const\}$, $\{y_i=\const\}$ and $\{w_i=\const\}$, respectively.

The transversality condition implies that $(x_1,x_2,y_1,y_2)$ is a system of local coordinates on $M$ and the matrices
$$
W_x=\left(\frac{\partial w_i}{\partial x_j}\right)_{i,j=1,2}, \qquad W_y=\left(\frac{\partial w_i}{\partial y_j}\right)_{i,j=1,2}
$$
have rank 2. We recall that $d_x$ and $d_y$ denote the exterior differential with respect to $x=(x_1,x_2)$ and $y=(y_1,y_2)$ coordinates, i.e. $d=d_x+d_y$. Then, it is easy to verify that the mapping
$$
(s:t)\mapsto \VV^\perp_{(s:t)}=\spn\{s\omega^1_i+t\omega^2_i\ |\ i=1,2\},
$$
where
$$
\omega^1_i=d_xw_i,\qquad\omega^2_i=d_yw_i,
$$
defines the unique Veronese embedding $\R P^1\to Gr_2(T^*M)$ such that $\VV^\perp_{(1:0)}$, $\VV^\perp_{(0:1)}$ and $\VV^\perp_{(1:1)}$ annihilate $T\FF_1$, $T\FF_2$ and $T\FF_3$, respectively. Therefore, the 3-web defines the following conformal metric
\begin{equation}\label{metric_form}
g=\omega_1^1\cdot\omega_2^2-\omega_2^1\cdot\omega_1^2
\end{equation}
which is necessarily hyper-Hermitian if the torsion of the Chern connection vanishes. By Theorem \ref{thm1} any hyper-Hermitian metric of the neutral signature can be locally written in this form. In order to write down the Chern connection and its torsion tensor in coordinates we denote
$$
C=W_xW_y^{-1}.
$$
Then
\begin{equation}\label{chern_form}
\begin{aligned}
\nablac_{\partial_{x_i}}\partial_{x_j}=\sum_{s,t=1,2}\partial_{x_i}(C^j_s){C^{-1}}^s_t\partial_{x_t},&\qquad
\nablac_{\partial_{y_i}}\partial_{y_j}=-\sum_{s,t=1,2}{C^{-1}}^j_s\partial_{y_i}(C^s_t)\partial_{y_t},\\
\nablac_{\partial_{x_i}}\partial_{y_j}=0,&\qquad
\nablac_{\partial_{y_i}}\partial_{x_j}=0,\\
\end{aligned}
\end{equation}
and one verifies that
$$
\nablac g=\omega\otimes g
$$
where
\begin{equation}\label{weyl_form}
\omega=\frac{d_x\det W_y}{\det W_y} + \frac{d_y\det W_x}{\det W_x}.
\end{equation}
This confirms Corollary \ref{cor1} and $([g],\nablac)$ is a Weyl structure indeed. The following two propositions describe the hyper-Hermitian case and are proved by simple calculations. Note that equation \eqref{tor_form1} below is a generalisation of the Hirota equation from \cite{DK1, Z}. 

\begin{proposition}\label{prop1}
The torsion of $\nablac$ vanishes if and only if
\begin{equation}\label{tor_form1}
\partial_{x_1}w_i\partial_{x_2}\det W_y-\partial_{x_2}w_i\partial_{x_1}\det W_y = \partial_{y_1}w_i\partial_{y_2}\det W_x- \partial_{y_2}w_i\partial_{y_1}\det W_x,
\end{equation}
for $i=1,2$.
\end{proposition}

\begin{proposition}\label{prop2}
If the torsion of $\nablac$ vanishes then the Ricci tensor of $\nablac$ is skew-symmetric and has the following form
$$
Ric(\nablac)=\left(\begin{array}{cc} 0&(R_{ij})_{i,j=1,2}\\ -(R_{ji})_{i,j=1,2}&0\end{array}\right)
$$
where
\begin{equation}\label{curv_form}
R_{ij}=Ric(\nablac)(\partial_{x_i},\partial_{y_j})= -Ric(\nablac)(\partial_{y_j},\partial_{x_i}) =\partial_{y_j}\frac{\partial_{x_i}(\det C)}{\det C}.
\end{equation}
\end{proposition}
Finally, we compute the Lee form $\AAA$ for the hyper-Hermitian structure defined by $(\FF_1,\FF_2,\FF_3)$
\begin{equation}\label{lee_form}
\AAA=\frac{d_x\det W_y}{\det W_y} + \frac{d_y\det W_x}{\det W_x}.
\end{equation}
Thus, it coincides with the Weyl form \eqref{weyl_form}.

\paragraph{Hyper-K\"ahler structures.}
We assume that $[g]$ is a hyper-Hermitian conformal metric of neutral signature on a 4-dimensional manifold. It is well known that the class $[g]$ contains a hyper-K\"ahler metric if and only if the Lee form of $g$ is closed \cite{B}. We shall prove that this hyper-K\"ahler condition can be expressed in term of the Chern connection. Indeed we have
\begin{proposition}\label{prop3}
$d\AAA=0$ if and only if $Ric(\nablac)=0$.
\end{proposition}
\begin{proof}
It follows from \eqref{lee_form} that
$$
d\AAA=d_y\left(\frac{d_x\det W_y}{\det W_y} - \frac{d_x\det W_x}{\det W_x}\right).
$$
On the other hand $\det C=\frac{\det W_x}{\det W_y}$ and it follows from \eqref{curv_form} that $Ric(\nablac)=0$ if and only if 
$$
\partial_{y_j}\left(\frac{\partial_{x_i}\det W_x}{\det W_x}-\frac{\partial_{x_i}\det W_y}{\det W_y}\right)=0.
$$
Thus, $d\AAA=0$ if and only if $Ric(\nablac)=0$ indeed.
\end{proof}

As a conclusion we get that in dimension 4 the Pleba\'nski equation derived in Section \ref{sec3} using the Chern connection describes the hyper-K\"ahler structures exactly as in \cite{P}. The Chern connection is the Levi-Civita connection for the hyper-K\"ahler metric in this case, because $\AAA=\omega$.

\paragraph{Acknowledgements.} The work has been partially supported by the Polish National Science Centre grant DEC-2011/03/D/ST1/03902.


\begin{thebibliography}{99}
\bibitem{AG0} M. A. Akivis, V. V. Goldberg, \textit{Semiintegrable almost Grassmann structures}, Differential Geometry and Its Applications, Vol. 10 (3) (1999), 257--294.
\bibitem{AG} M. A. Akivis, V. V. Goldberg, \textit{Differential geometry of webs}, In: Handbook of Differential Geometry, Vol. I, North-Holland, Amsterdam, 2000, 1--152, Chapter~1.
\bibitem{AMT} D. Alekseevsky, C. Medori, A. Tomassini, \textit{Homogeneous para-K\"ahler Einstein manifolds}, Russ. Math. Surv. Vol. 64, No. 1 (2009).
\bibitem{B} C. Boyer, \textit{A note on hyperhermitian four-manifolds}, Proc. Amer. Math. Soc. 102 (1988),
157--164.
\bibitem{Ca} D. Calderbank \textit{Integrable Background Geometries}, SIGMA 10 (2014).
\bibitem{CDT} S. Casey, M. Dunajski, P. Tod, \textit{Twistor geometry of a pair of second order ODEs}, Commun. Math Phys, Vol. 321 (2013), 681--701.
\bibitem{C} S.-S. Chern, \textit{Sur la g\'eom\'etrie d'un syst\'eme d'\'equations diff\'erentialles du second ordre}, Bull.~Sci.~Math. 63 (1939), 206--212.
\bibitem{De} A. Derdzinski, \textit{Connections with skew-symmetric Ricci tensor on surfaces}, Results Math. 52, No. 3-4 (2008), 223--245.
\bibitem{D2} M. Dunajski, \textit{A class of Einstein-Weyl spaces associated to an integrable system of hydrodynamic type}, J. Geom. Phys. 51 (2004), 126--137.
\bibitem{DK1} M. Dunajski, W. Kry\'nski, \textit{Einstein--Weyl geometry, dispersionless Hirota equation and Veronese webs},  Math. Proc. Camb. Phil. Soc., Vol. 157, Issue 01 (2014), 139--150.
\bibitem{DK2} M. Dunajski, W. Kry\'nski, \textit{Point invariants of third-order ODEs and hyper-CR Einstein-Weyl structures}, J. Geom. Phys., Vol. 86 (2014) 296--302.
\bibitem{DW1} M. Dunajski, S. West, \textit{Anti-self-dual conformal structures with null Killing vectors from projective structures}, Commun. Math. Phys. 272 (2007), 85--118.
\bibitem{DW2} M. Dunajski, S. West, \textit{Anti-self-dual conformal structures in neutral signature}, Recent developments in pseudo-Riemannian geometry, 113–148, ESI Lect. Math. Phys., Eur. Math. Soc., Zürich, 2008. 
\bibitem{FK} E. Ferapontov, B. Kruglikov, \textit{Dispersionless integrable systems in 3D and Einstein-Weyl geometry}, J. Differential Geom. Vol. 97, No. 2 (2014), 215--254.
\bibitem{FP} J. Finley, J. Pleba\'nski, \textit{Further heavenly metrics and their symmetries}, J. Math. Phys., Vol. 17, No. 4 (1976), 585--596.
\bibitem{GZ} I. M. Gelfand, I. Zakharevich, \textit{Webs, Veronese curves, and bi-Hamiltonian systems}, Journal of Functional Analysis, vol. 99, no. 1 (1991), 150--178.
\bibitem{G} D. A. Grossman, \textit{Torsion-free path geometries and integrable second order ODE systems}, Selecta Mathematica, Volume 6, Issue 4, (2000), 399--342. 
\bibitem{JK} B. Jakubczyk, W. Kry\'nski, \textit{Vector fields with distributions and invariants of ODEs}, Journal of Geometric Mechanics, Vol. 5, No. 1 (2013).
\bibitem{KM} B. Kruglikov, O. Morozov, \textit{SDiff(2) and uniqueness of the Plebański equation},  Journal of Mathematical Physics, Vol. 53, No. 8 (2012).
\bibitem{K1} W. Kry\'nski, \textit{Geometry of isotypic Kronecker webs}, Central European Journal of Mathematics, Vol. 10 (2012) 1872--1888.
\bibitem{K2} W. Kry\'nski, \textit{Webs and projective structures on a plane}, Differential Geometry and Its Applications, Vol. 37 (2014) 133--140. 
\bibitem{K3} W. Kry\'nski, \textit{Paraconformal structures, ODEs and totally geodesic manifolds}, submitted, arXiv:1310.6855 (2013).
\bibitem{M} T. Mettler, \textit{Reduction of $\beta$-integrable 2-Segre structures}, Comm. Anal. Geom., Vol. 21, No. 2 (2013), 331--353.
\bibitem{N} P. Nagy, \textit{Webs and curvature}, In: Web Theory and Related Topics, Toulouse, December, 1996, World Scientific Publishing, River Edge, 2001, 48--91.
\bibitem{Pa} A. Panasyuk, \textit{Veronese webs for bi-Hamiltonian structures of higher corank}, In: Poisson Geometry, Warsaw, August 3–15, 1998, Banach Center Publ., 51 (2000) 251--261.
\bibitem{Pe} R. Penrose, \textit{Nonlinear gravitons and curved twistor theory}, General Relativity and Gravitation 7 (1976), 31--52.
\bibitem{P} J. Pleba\'nski, \textit{Some solutions of complex Einstein equations}, J. Math. Phys., Vol 16, No. 12 (1975), 2395--2402.
\bibitem{R} M. Randall, \textit{Local obstructions to projective surfaces admitting skew-symmetric Ricci tensor}, J. Geom. Phys. 76 (2014), 192--199.
\bibitem{Z} I. Zakharevich, \textit{Nonlinear wave equation, nonlinear Riemann problem, and the twistor transform of Veronese webs}, arXiv:math-ph/0006001 (2000).
\bibitem{Z1} I. Zakharevich, \textit{Kronecker webs, bihamiltonian structures and the method of argument translation}, Transform. Groups, Vol. 6, No. 3 (2001), 267--300.
\end{thebibliography}
\end{document}